\documentclass[12pt]{article}
\usepackage[english]{babel}
\usepackage[latin1]{inputenc}
\usepackage{amsfonts,amssymb,amsmath, epsfig}
\usepackage{color,graphicx,graphics,psfrag}
\usepackage{amsthm, amstext,amscd}
\usepackage{pifont}% http://ctan.org/pkg/pifont
\usepackage[toc,page]{appendix}
\usepackage{dsfont}
\providecommand{\keywords}[1]{\textbf{\textit{Keywords:}} #1}
\usepackage{cases}
\usepackage{caption, subcaption}
\usepackage{authblk} % for the addresses

\def\Box{\leavevmode\vbox{\hrule
     \hbox{\vrule\kern4pt\vbox{\kern4pt}%
           \vrule}\hrule}}

\def\paragraph#1{{\bf #1\ }}

\newtheorem{lemma}{Lemma}[section]

\newtheorem{theorem}[lemma]{Theorem}

\newtheorem{definition}[lemma]{Definition}

\newtheorem{proposition}[lemma]{Proposition}

\newtheorem{remark}{Remark}[section]
\newcommand{\cmark}{\ding{51}}

\newcommand{\xmark}{\ding{55}}
\newcommand{\XX}{\mathbf{X}}
\newcommand{\bXX}{\bar{\mathbf{X}}}
\newcommand{\tXX}{\tilde{\mathbf{X}}}
\newcommand{\VV}{\mathbf{V}}
	
\newcommand{\tVV}{\tilde{\mathbf{V}}}	

\newcommand{\RR}{\mathbb{R}}
\newcommand{\NN}{\mathbb{N}}
\newcommand{\LL}{\mathcal{L}}
\newcommand{\PP}{\mathcal{P}}
\newcommand{\OO}{\mathcal{O}}
\newcommand{\llambda}{\boldsymbol{\lambda}}
\newcommand{\bllambda}{\bar{\boldsymbol{\lambda}}}
\newcommand{\mmu}{\boldsymbol{\mu}}
\newcommand{\tmmu}{\tilde{\boldsymbol{\mu}}}

\title{Damped Arrow-Hurwicz algorithm for sphere packing}
\author[1]{Pierre Degond}
\author[1]{Marina A. Ferreira}
\author[2]{Sebastien Motsch}
\affil[1]{Department of Mathematics, South Kensington Campus, Imperial College London, SW7 2AZ London, United Kingdom}
\affil[2]{School of Mathematical \& Statistical Sciences, Arizona State University, Tempe, AZ 85287-1804, United States of America}

\begin{document}

\maketitle

\setcounter{equation}{0}

\begin{abstract}
We consider algorithms that, from an arbitrarily sampling of $N$ spheres (possibly overlapping), find a close packed configuration without overlapping. 
These problems can be formulated as minimization problems with non-convex constraints. 
For such packing problems, we observe that the classical iterative Arrow-Hurwicz algorithm does not converge.
We derive a novel algorithm from a multi-step variant of the Arrow-Hurwicz scheme with damping.
We compare this algorithm with classical algorithms belonging to the class of linearly constrained Lagrangian methods and show that it performs better.
We provide an analysis of the convergence of these algorithms in the simple case of two spheres in one spatial dimension. 
Finally, we investigate the behaviour of our algorithm when the number of spheres is large.

\end{abstract}


\tableofcontents

\section{Introduction}

Particle packing problems can be encountered in many different systems, from the formation of planets or cells in live tissues to the dynamics of crowds of people. 
In particular, they have been widely investigated in the study of granular media~\cite{Makse2004}, glasses~\cite{Zallen1983} and liquids~\cite{Hansen1990}. 
More recently, particle packings have revealed to be important tools in biology~\cite{kreft1998bacteria} and social sciences~\cite{cordoba2009crowd}.

Packing problems give rise to NP-hard non-convex optimization problems~\cite{Hifi2009} and the optimal solution is in general not unique, since permutations, rotations or reflections may generate equivalent solutions. 
We refer the reader to~\cite{Hifi2009} for a review on packing problems. 
In the literature, one can find numerical studies involving particles with various shapes such as ellipses~\cite{Wenxiang2010} or even non-convex particles~\cite{faure2009spheres}.
However, in the present work we assume that the particles are simply identical spheres with diameter $d$ in $\RR^2$, 
but the methodology is general and will be extended to other cases in future work.  
We consider algorithms that, given an initial configuration of $N$ spheres (possibly overlapping), find a nearby packed configuration without overlapping. Indeed, in many natural systems individuals or particles only seek to achieve a locally optimal solution. 
Therefore, it is more likely that they reach a local configuration that does not necessarily correspond to a global optimum.
By combining our method with, for example, simulated annealing techniques~\cite{aarts1988annealing}, we could convert our algorithms into global minimum search algorithms. It is however not the objective we pursue here. 

Classical procedures to solve non-convex minimization problems include Uzawa-Arrow-Hurwicz type algorithms~\cite{Uzawa58}, augmented Lagrangian~\cite{Bertsekas1996, Birgin2005}, linearly constrained Lagrangian (LCL), sequential quadratic programming (SQP)~\cite{nocedal2006numerical}, among others.
The SQP and the Uzawa-Arrow-Hurwicz algorithms are widely used. However they require the Hessian matrix of the function to be minimized to be positive definite, which is not always the case in this type of problems (see the example presented in section~\ref{sec:linearAnalysis}).
In general, all these methods perform well with a small number of particles. However we are interested in the case where this number becomes large.

In~\cite{doye2006atomicClusters, hartke2006atomicClusters} the authors study the shape of three dimensional clusters of atoms under the effect of soft potentials by using molecular dynamics. This approach differs from ours with regard to the non-overlapping constraints, which are approximated by soft potentials, producing soft dynamics. 
Although being more costly when dealing with a large number of particles, we have opted by the hard dynamics approach, since it allows for a higher precision in the treatment of the constraints. This proves effective when dealing with interaction between rigid bodies, where the effect of the rigid boundaries plays an important role. This motivates the present work.

We start in section~\ref{sec:minimizationProblems} by presenting two formulations of the problem. The first one is the classical minimization approach. The second one considers a constrained dynamical system in the spirit of~\cite{Maury2006}.
We also present two equivalent types of non-overlapping constraints involving smooth or non-smooth functions which are found in the literature~\cite{Hifi2009, Addis2008}.
To solve the non-convex minimization problems arising in these formulations, we first consider in section~\ref{sec:ArrowHurwicz} the classical Arrow-Hurwicz algorithm (AHA). 
In section~\ref{sec:DAHA} we introduce a novel multi-step scheme based on a second-order ODE interpretation of the minimization problem: the damped Arrow-Hurwicz algorithm (DAHA). 
We test the DAHA against two methods taken from the widely known class of linearly constrained Lagrangian algorithms~\cite{nocedal2006numerical,Maury2006}.
These algorithms consist of a sequence of convex minimization problems, for we refer to them as \textit{nested algorithms (NA)} and they shall be referred to as the NAP and NAV.
The convergence of the four algorithms (the AHA, DAHA, NAP and NAV) is analyzed in section~\ref{sec:linearAnalysis} for the case of two spheres in one  dimension. In section~\ref{sec:numerics} the algorithms are numerically compared for the cases of many spheres in two dimensions. Finally, conclusions and future works are presented in section~\ref{sec:future}.

\section{The damped Arrow-Hurwicz algorithm  (DAHA)}

\subsection{Minimization problems for sphere packing}\label{sec:minimizationProblems}

We first recall two different formulations of generic minimization problems.
Let $N$ and $b$ be two given positive integers.
We consider first the problem of finding a configuration $\bXX$ such that
\begin{eqnarray}
\bXX \in \underset{\phi_{k\ell}(\XX)\leq 0,\ k,\ell=1,...,N,\ k < \ell}{\mbox{argmin}} W(\XX).\label{eq:GenericProblem}
\end{eqnarray}
where $W: \RR^{bN} \rightarrow \RR$ is a convex function (not necessarily strictly convex). The functions $\phi_{k\ell}: \RR^{bN} \rightarrow \RR,\ k,\ell=1,...,N,\ k<\ell$ are continuous but not necessarily convex.
In what follows, $d$ will denote the diameter of a sphere, $N$ the number of spheres, $b$ the spatial dimension, $\XX$ the position of the center of the spheres and $\phi_{k\ell}$ the non-overlapping constraint functions between the $kth$ and $\ell th$ spheres.
An illustration of the non-overlapping constraints, as well as, a possible solution for $N=7$ are presented in Figure~\ref{fig:illustration_Solution}.
\begin{figure}[!h]
\begin{subfigure}{.5\textwidth}
  \centering
  \includegraphics[scale=0.8]{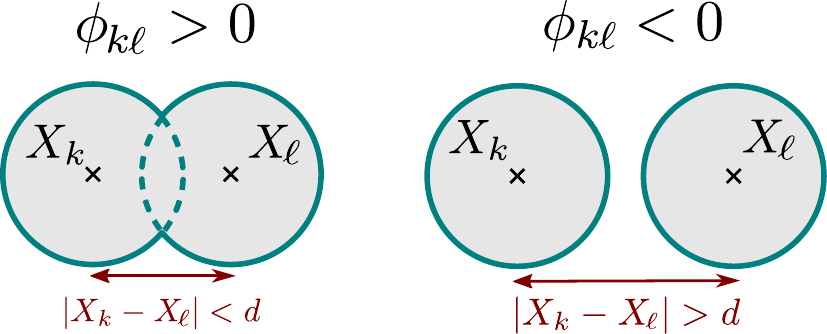}
   \caption{}
  \label{fig:illustration1}
\end{subfigure}
\begin{subfigure}{.5\textwidth}
  \centering
  \includegraphics[scale=0.8]{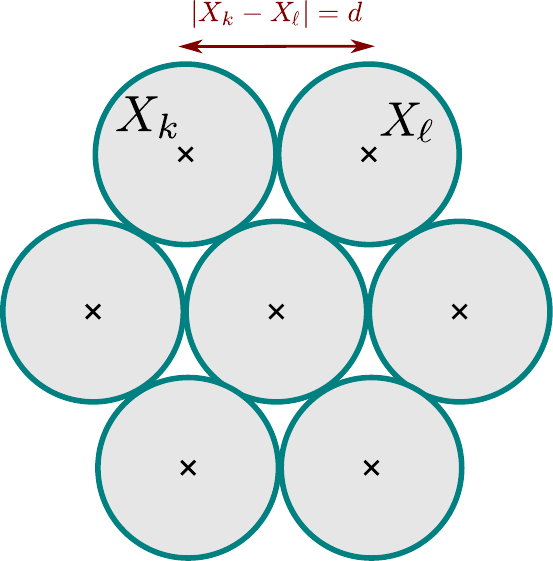}
    \caption{}
  \label{fig:illustration2}
\end{subfigure}
\caption{Representation of the non-overlapping constraints,~\ref{fig:illustration1}, and a possible optimal solution of~\eqref{eq:GenericProblem} for $N=7$,~\ref{fig:illustration2}.}
\label{fig:illustration_Solution}
\end{figure}
We suppose that the set $\{\XX \in \RR^{bN} | \phi_{k\ell}(\XX) \leq 0 ,\ k,\ell=1,...,N,\ k<\ell\}$ is compact.
In these conditions, $\bXX$ exists but may not be unique. 
We also assume that $\phi_{k\ell},\ k,\ell=1,...,N,\ k<\ell$ and $W$ are $C^1$ functions in the neighbourhood of $\bXX$. 

We now present a second formulation consisting in solving a minimization problem associated with a discrete dynamical system which has $\bXX$ as a fixed point. 
Let $| \cdot |$ denote the Euclidean norm on $\RR^{b}$. The problem is formulated iteratively: given an initial configuration $\XX^0 = \{X_i^0\}_{i=1,...,N}$, we pass from iterate $\XX^p$ to iterate $\XX^{p+1}$ as follows
\begin{numcases}{}
\XX^{p+1}         = \XX^p + \tau\  \VV^{p+1}\label{eq:vel1}\\
\VV^{p+1}   \in \underset{\phi_{k\ell}(\XX^p+\tau\   \VV)\leq 0,\  k,\ell=1,...,N,\ k < \ell}{\mbox{argmin}} \frac{1}{2} \sum_{i=1}^N |V_i + \nabla_{X_i}W(\XX^p)|^2,\label{eq:vel2}
\end{numcases}
where $\tau>0$ is a given parameter and $\VV = \{V_i\}_{i=1,...,N}$.
We define $\tXX$ as a fixed point of this problem. 
Consequently, $\tXX$ satisfies
\begin{eqnarray}
0   &\in & \underset{\phi_{k\ell}(\tXX+\tau\   \VV)\leq 0,\ k,\ell=1,...,N,\ k < \ell}{\mbox{argmin}} \frac{1}{2}\sum_{i=1}^N |V_i+ \nabla_{X_i}W(\tXX)|^2,\label{eq:vel_eq}
\end{eqnarray}
Note that the set defined by the constraints is compact, so the minima of~\eqref{eq:vel_eq} exist, but may not be unique.

The minimization problem~\eqref{eq:GenericProblem} can be formulated in terms of the Lagrangian $\LL: \RR^{bN}\times (\RR^+_0)^{N(N-1)/2} \rightarrow \RR$ defined by
$$\LL(\XX, \llambda) =  W(\XX) + \sum\limits_{k,\ell \in \{1,...,N\},\ k<\ell} \lambda_{k\ell} \phi_{k\ell}(\XX),$$
where $\llambda = \{ \lambda_{k\ell} \}_{k,\ell=1,...,N,\ k < \ell} $ represents the set of Lagrange multipliers.
If $\bXX$ is a solution of the minimization problem~\eqref{eq:GenericProblem}, then, there exists $\bllambda\in (\RR_0^+)^{N(N-1)/2}$ such that $(\bXX,\bllambda)$ is a critical-point of the Lagrangian, namely, $(\bXX,\bllambda)$ satisfies the {\it KKT-conditions}~\cite{Kjeldsen2000}:
$$
\begin{cases}
\nabla_{X_i} \LL(\bXX,\bllambda) =\ 0,\ i=1,...,N  \\
\left(\nabla_{\lambda_{k\ell}} \LL(\bXX,\bllambda) =0\ \text{ and }\bar\lambda_{k\ell}\geq 0\right) \text{ or } \left(\nabla_{\lambda_{k\ell}} \LL(\bXX,\bllambda) <0 \text{ and } \bar\lambda_{k\ell} = 0\right), \\
\hspace{8cm} k,\ell = 1,...,N, k < \ell
\end{cases}
$$
which is equivalent to
\begin{equation}
\begin{cases}
\nabla_{X_i} W(\bXX) + \sum\limits_{k,\ell \in \{1,...,N\},\ k<\ell} \bar\lambda_{k\ell}\nabla_{X_i}\phi_{k\ell}(\bXX)  =\ 0,\ i=1,...,N  \\
\left(\phi_{k\ell}(\bXX) =0\text{ and }\bar\lambda_{k\ell}\geq 0\right) \text{ or } \left(\phi_{k\ell}(\bXX) < 0 \text{ and } \bar\lambda_{k\ell} = 0\right),\\
\hspace{6cm} k,\ell=1,...,N,k <\ell.
\end{cases}\label{eq:systemSaddlePoint}
\end{equation}
We have reduced our original problem~\eqref{eq:GenericProblem} to a critical-point system, with a possible enlargement of the set of solutions.
Contrarily to convex optimization, in the case of packing problems, these critical-points may not be saddle-points~\cite{PreprintSebastien}.

We also formulate the minimization problem~\eqref{eq:vel2} in terms of a Lagrangian $\LL^p$,
$$
\LL^p( \VV, \mmu) = \frac{1}{2}\sum_{i=1}^N |V_i+ \nabla_{X_i}W(\XX^p)|^2 + \sum\limits_{k,\ell \in \{1,...,N\},\ k<\ell} \mu_{k\ell}\phi_{k\ell}(\XX^p + \tau\  \VV),$$
where $\mmu = \{\mu_{k\ell}\}_{k,\ell=1,...,N,\ k< \ell}$ is the set of Lagrange multipliers associated to the constraints.
The gradients of the Lagrangian are given by
\begin{eqnarray}
\nabla_{V_i} \LL^p(\VV,\mmu) &=& V_i + \nabla_{X_i}W(\XX^p) \nonumber\\
&&\ \ \ \  +\ \tau\ \sum\limits_{k,\ell \in \{1,...,N\},\ k<\ell} \mu_{k\ell} \nabla_{X_i} \phi_{k\ell}(\XX^p + \tau\  \VV), \ i=1,...,N\nonumber \\
\nabla_{\mu_{k\ell}} \LL^p(\VV,\mmu) &=&\phi_{k\ell}(\XX^p + \tau\  \VV),\ k,\ell = 1,...,N.\nonumber
\end{eqnarray}
The dynamical system is written:
$\tXX^{p+1} = \tXX^p + \tau\ \tVV^{p+1}$ such that $(\tVV^{p+1},\tmmu^{p+1})$ is a solution of the critical-point problem
\begin{eqnarray}
\begin{cases}
 \tilde V_i^{p+1}+ \nabla_{X_i}W(\tXX^p)  +\tau\ \sum\limits_{k,\ell \in \{1,...,N\},\ k<\ell} \tilde\mu_{k\ell}^{p+1} \nabla_{X_i} \phi_{k\ell}(\tXX^p + \tau\  \tVV^{p+1}) = 0,\\
 \hspace{9cm} \ i=1,...,N \\
\left(\phi_{k\ell}(\tXX^p + \tau\  \tVV^{p+1}) =0\text{ and }\tilde\mu_{k\ell}^{p+1}\geq 0\right) \text{ or } \\ 
\hspace{1cm}\left( \phi_{k\ell}(\tXX^p + \tau\ \tVV^{p+1}) < 0 \text{ and } \tilde \mu_{k\ell}^{p+1} = 0\right),\ k,\ell=1,...,N,\ k<\ell
\end{cases}\label{eq:systemSaddlePoint_V}
\end{eqnarray}
Likewise, the fixed point $\tXX$ of the dynamical system is defined such that there exists $\tmmu$ such that 
\begin{eqnarray}
\begin{cases}
\nabla_{X_i}W(\tXX) +\tau\ \sum\limits_{k,\ell \in \{1,...,N\},\ k<\ell} \tilde\mu_{k\ell} \nabla_{X_i} \phi_{k\ell}(\tXX) = 0,\ i=1,...,N \\
\left(\phi_{k\ell}(\tXX) =0\text{ and }\tilde\mu_{k\ell}\geq 0\right) \text{ or } \left( \phi_{k\ell}(\tXX) < 0 \text{ and } \tilde \mu_{k\ell} = 0\right),\\
\hspace{2cm}k,\ell=1,...,N,\ k<\ell.
\end{cases}\label{eq:fixedPoint_V}
\end{eqnarray}

Then, it is clear that problems~\eqref{eq:systemSaddlePoint} and~\eqref{eq:fixedPoint_V} are equivalent for all values of $\tau>0$ by setting $\bllambda = \tau \tmmu$.
However, the choice of $\tau$ is important to ensure convergence of the dynamical system~\eqref{eq:systemSaddlePoint_V} to the fixed point.

As it will be obvious below, all functions $W$ and $\phi_{k\ell}$ used throughout the paper will satisfy the conditions considered in this section.
The nonlinear systems~\eqref{eq:systemSaddlePoint} or~\eqref{eq:systemSaddlePoint_V} will have to be solved by an iterative algorithm.
We now present the algorithms considered in the paper.

\subsection{The Arrow-Hurwicz algorithm (AHA)}\label{sec:ArrowHurwicz}

The classical Arrow-Hurwicz iterative algorithm~\cite{Uzawa58} searches a saddle-point of the Lagrangian by alternating steps in the direction of $-\nabla_{\XX}\LL$ and $+\nabla_{\llambda}\LL$. 
Using this idea, a saddle-point is then a steady-state solution of the Arrow-Hurwicz system of ODE's (AHS) which is defined next.

\begin{definition}
The \textbf{Arrow-Hurwicz system (AHS)} is defined by
\begin{numcases}{}
\dot X_i           & $=\ -\alpha \left( \nabla_{X_i} W(\XX) + \sum\limits_{k,\ell \in \{1,...,N\},\ k<\ell} \lambda_{k\ell}\nabla_{X_i}\phi_{k\ell}(\XX)\right),$ \nonumber\\
& $\hspace{6cm} \ i=1,...,N$ \label{eq:AHSystem1} \\
\dot \lambda_{k\ell}  &$=\ 
\begin{cases}
0,\ \text{ if } \lambda_{k\ell}=0 \text{ and } \phi_{k\ell}(\XX)<0\\
\beta \phi_{k\ell}(\XX),\text{ otherwise}
\end{cases}$,\nonumber \\
& $\hspace{5cm} k,\ell = 1,...,N,\ k< \ell,$ \label{eq:AHSystem2}
\end{numcases}
where $\alpha$ and $\beta$ are positive constants. 
Considering a small time-step $\Delta t$, a semi-implicit Euler discretization scheme of the previous system leads to the \textbf{Arrow-Hurwicz algorithm (AHA)}, which is defined iteratively by
\begin{eqnarray}
\left\{
\begin{array}{ll}
X_i^{n+1}          &=\ X_i^n - \alpha \left[\nabla_{X_i} W(\XX^n) + \sum\limits_{k,\ell \in \{1,...,N\},\ k<\ell} \lambda_{k\ell}^n\nabla_{X_i}\phi_{k\ell}(\XX^n)\right],\\
&\hspace{8cm} \ i=1,...,N \\
\lambda_{k\ell}^{n+1} &=\ \max\{0,\lambda_{k\ell}^n + \beta\phi_{k\ell}(\XX^{n+1})\},\ k,\ell = 1,..., N,\ k< \ell 
\end{array} \right. \label{eq:uz}
\end{eqnarray}
where $\alpha$ and $\beta$ now correspond to $\tilde \alpha =\alpha \Delta t\ $ and $\tilde \beta =\beta \Delta t$ and the tildes have been dropped for simplicity.
\end{definition}

The original AHA was formulated using a fully explicit Euler scheme, but it has proved more accurate to use a semi-implicit scheme.
Finding a local steady-sate solution of \eqref{eq:AHSystem1}-\eqref{eq:AHSystem2} in the case of a packing problem has revealed not to be always possible because it often happens that no critical-point is a saddle-point~\cite{PreprintSebastien}. This manifests itself by the existence of periodic solutions of the AHS which do not converge to the critical-point. 
In order to overcome this difficulty we propose the damped Arrow-Hurwicz algorithm which is presented next. 
This method is based on a modification of the dynamics of the AHS that transforms an unstable critical-point into an asymptotically stable one.
The performance of our method will be tested by comparing with previous approaches~\cite{nocedal2006numerical, Maury2006}, which are based on a modification of the Lagrangian by linearly approximating the constraints. 
These approaches are presented in section~\ref{sec:NestedAlg}.

\subsection{The damped Arrow-Hurwicz algorithm}\label{sec:DAHA}

In order to avoid periodic solutions we will add a damping term as described below. Note that we are not interested on the transient dynamics of the system, but rather on its asymptotic behaviour. 

We propose the following definition.

\begin{definition}
We define the \textbf{damped Arrow-Hurwicz system (DAHS)} as
\begin{numcases}{}
 \ddot X_i &$= - \alpha^2 [\nabla_{X_i}  W (\XX) + \sum\limits_{k,\ell \in \{1,...,N\},\ k<\ell} \lambda_{k\ell} \nabla_{X_i} \phi_{k\ell}(\XX)]$ \nonumber\\ 
& \ \ \ $ - \alpha \beta \sum\limits_{k,\ell \in \{1,...,N\},\ k<\ell} \phi_{k\ell}(\XX)  \lambda_{k\ell} \nabla_{X_i} \phi_{k\ell}(\XX)  - c \dot X_i\ ,$\nonumber\\
& \hspace{5cm}$ \ i = 1,...,N$ \label{eq:secc1}\\
\dot \lambda_{k\ell}  &$=\ 
\begin{cases}
0,\ \text{ if } \lambda_{k\ell}=0 \text{ and } \phi_{k\ell}(\XX)<0\\
\beta \phi_{k\ell}(\XX),\ \text{ otherwise }
\end{cases}$,\nonumber \\
& \hspace{3.5cm}$ k,\ell = 1,...,N,\   k< \ell$ \label{eq:secc2} 
\end{numcases}
where $\alpha, \beta$ and $c$ are positive constants and
the~\textbf{damped Arrow-Hurwicz algorithm (DAHA)} as the corresponding semi-implicit discrete scheme:
\begin{numcases}{}
 X_i^{n+1} &$ = 
 \frac{1}{1+c/2} \left( 2X_i^n  - (1-c/2) X_i^{n-1}\right)$\nonumber\\
& $-  \frac{\alpha^2}{1+c/2} [\nabla_{X_i}  W (\XX^n) + \sum\limits_{k,\ell \in \{1,...,N\},\ k<\ell} \lambda_{k\ell}^n \nabla_{X_i} \phi_{k\ell}(\XX^n)]$ \nonumber\\ 
&  $- \frac{\alpha\beta}{1+c/2} \sum\limits_{k,\ell \in \{1,...,N\},\ k<\ell} \phi_{k\ell}(\XX^n)  \lambda_{k\ell}^n \nabla_{X_i} \phi_{k\ell}(\XX^n),$\nonumber \\
&\hspace{6cm}\ \ $ i = 1,...,N $\label{eq:secc_discrete1}\\
\lambda_{k\ell}^{n+1}  &$ = \max\{0,\lambda_{k\ell}^n + \beta\phi_{k\ell}(\XX^{n+1})\},\ k,\ell = 1,..., N,\ k< \ell,$\label{eq:secc_discrete2} 
\end{numcases}
where $\alpha, \beta$ and $c$ correspond now to numerical parameters.
\end{definition}

Note that the DAHA is a multi-step scheme, since not only one, but two previous configurations $\XX^{n-1}$ and $\XX^n$ are used to obtain $\XX^{n+1}$.
By setting $c=2$, the method is reduced to a one-step method.

In the following we present the derivation of the DAHS.
We start by considering the AHS~\eqref{eq:AHSystem1}-\eqref{eq:AHSystem2} presented in the previous section.
We then take the second-order version of \eqref{eq:AHSystem1}. 
For each $i=1,...,N$ we have 
\begin{eqnarray}
 \ddot X_i = && - \alpha \sum_{m=1}^N \nabla_{X_m} \big[  \nabla_{X_i}  W (\XX) + \sum\limits_{k,\ell \in \{1,...,N\},\ k<\ell} \lambda_{k\ell} \nabla_{X_i} \phi_{k\ell}(\XX) \big] \dot X_m \nonumber \\
&&- \alpha \sum\limits_{k,\ell \in \{1,...,N\},\ k<\ell} \dot \lambda_{k\ell} \nabla_{X_i} \phi_{k\ell}(\XX) . \label{eq:sec2}
\end{eqnarray}
Using (\ref{eq:AHSystem2}), we can replace $\dot \lambda_{k\ell}$ in (\ref{eq:sec2}) by $\beta \phi_{k\ell}(\XX) H(\lambda_{k\ell})$, where $H$ is the Heaviside function. 
It also proves more efficient to modify it by replacing $H(\lambda_{k\ell})$ by $\lambda_{k\ell} H(\lambda_{k\ell})$, which is equal to $\lambda_{k\ell}$. 
We get
\begin{eqnarray}
 \ddot X_i =  - \alpha \sum_{m=1}^N \nabla_{X_m} \big[  \nabla_{X_i}  W (X) + \sum\limits_{k,\ell \in \{1,...,N\},\ k<\ell} \lambda_{k\ell} \nabla_{X_i} \phi_{k\ell}(\XX) \big] \dot X_m \label{eq:complicatedExpression} \\
  - \alpha \beta \sum\limits_{k,\ell \in \{1,...,N\},\ k<\ell} \phi_{k\ell}(\XX)  \lambda_{k\ell} \nabla_{X_i} \phi_{k\ell}(\XX). 
\end{eqnarray}
It turns out that passing to the second-order introduces exponentially growing modes (see Remark~\ref{rem:exponentialModes}).
\begin{remark}
Consider the simple ODE $\dot u = - \alpha u$ whose solution is $u(t)=u_0 e^{-\alpha t}$, where $u_0$ is the initial configuration. Differentiating both sides of the equation and substituting $\dot u$ by  $- \alpha u$ yields $\ddot u = \alpha^2 u$, whose solution includes now an exponentially growing mode: $u(t)=c_1 e^{-\alpha t}+c_2e^{\alpha t}$, where $c_1$ and $c_2$ are real constants determined by the initial configurations.
\label{rem:exponentialModes} \end{remark}
In order to remove these modes, we replace the term in~\eqref{eq:complicatedExpression}
by a simple second-order dynamics in the force field given by the right hand side of (\ref{eq:AHSystem1}). 
We get: 
\begin{eqnarray}
\ddot X_i = && - \alpha^2 \left[ \nabla_{X_i}  W (\XX) + \sum\limits_{k,\ell \in \{1,...,N\},\ k<\ell} \lambda_{k\ell} \nabla_{X_i} \phi_{k\ell}(\XX)  \right]\nonumber \\ 
&& - \alpha \beta \sum\limits_{k,\ell \in \{1,...,N\},\ k<\ell} \phi_{k\ell}(\XX)  \lambda_{k\ell} \nabla_{X_i} \phi_{k\ell}(\XX). \label{eq:DAHAsystemDerivation}
\end{eqnarray}
Now, we just add a velocity damping term in the form of $- c \dot X_i$ and we finally obtain~\eqref{eq:secc1}. 
We end up with the system~\eqref{eq:secc1}-\eqref{eq:secc2}.

\begin{remark}
We can interpret the first term, at the right hand side of~\eqref{eq:DAHAsystemDerivation} as a second-order dynamics version of~\eqref{eq:AHSystem1}.
Denoting by $T_1$ and $T_2$ the terms in~\eqref{eq:DAHAsystemDerivation} which are multiplied by $-\alpha^2$ and $-\alpha \beta$, respectively, we recover~\eqref{eq:AHSystem1} in an over-damped limit $\epsilon \left[ \ddot X_i + \alpha\beta T_2 \right] = -\alpha^2 T_1 - c \dot X_i$, with $\epsilon \to 0$ and $c=1$.
\end{remark}

\begin{proposition}
The AHS~\eqref{eq:AHSystem1}-\eqref{eq:AHSystem2} and the DAHS~\eqref{eq:secc1}-\eqref{eq:secc2} have the same equilibrium solutions.
\end{proposition}
\begin{proof}
If $(\llambda^*,\XX^*)$ is a steady state of the AHS, then either $ \phi_{k\ell}(\XX^*)= 0$ or $\lambda_{k\ell}^* = 0$. 
Consequently, $\lambda_{k\ell} \phi_{k\ell}(\XX^*) = 0$, which implies that the second part of equation~\eqref{eq:secc1} is null and  $\ddot \XX^* = 0$. 
Using a similar argument we conclude that a steady state of DAHS is also a steady state of AHS. 
\end{proof}

\subsection{Previous approaches}\label{sec:NestedAlg}

A common approach to solve the generic minimization problems~\eqref{eq:GenericProblem} and~\eqref{eq:vel1}-\eqref{eq:vel2} is based on the linearization of the constraint functions $\phi_{k\ell}$ around a certain configuration $\XX^p$, which we denote by $\phi_{k\ell}^p(\XX)$, i.e, 
\begin{eqnarray}
\phi_{k\ell}^p(\XX) = \phi_{k\ell}(\XX^p) + \nabla_{\XX} \phi_{k\ell}(\XX^p) \cdot (\XX - \XX^p).\label{eq:linearizationConstraints}
\end{eqnarray}
The solution $\XX^{p+1}$ of the resulting linearly constrained optimization problem is used to improve the linearization of the constraint functions and this process is iterated until convergence.
Note that this transformation turns the non-convex minimization problems~\eqref{eq:GenericProblem} and~\eqref{eq:vel1}-\eqref{eq:vel2} into a sequence of convex problems, for which there are many tools available~\cite{bertsekas1999nonlinear}.
We have chosen the Arrow-Hurwicz algorithm, however, any other method for convex optimization problems would suit our purpose.

This method belongs to the class of linearly constrained Lagrangian (LCL) methods~\cite{nocedal2006numerical} which have been used for large constrained optimization problems.

\subsubsection{The nested algorithm for the positions (NAP)}
 
Consider the system~\eqref{eq:systemSaddlePoint} with linearized constraint functions.
We propose the following definition.

\begin{definition}[\textbf{Nested Algorithm for the Positions (NAP)}]  Let $(\XX^p, \llambda^p)$ be given.
Define $\XX^{p,0}= \XX^p$, $\llambda^{p,0}= \llambda^p$ and $\phi_{k\ell}^p$ as in~\eqref{eq:linearizationConstraints}.
For a given $(\XX^{p,n}, \llambda^{p,n})$, let the step of the inner-loop be defined as
\begin{numcases}{}
X_i^{p,n+1}          &$= X_i^{p,n} - \alpha \left[\nabla_{X_i}W(\XX^{p,n})  + \sum\limits_{k,\ell \in \{1,...,N\},\ k<\ell} \lambda_{k\ell}^{p,n} \nabla_{X_i}\phi_{k\ell}^p(\XX^{p,n})\right],$\nonumber\\
&$\hspace{6cm} \ i=1,...,N$ \label{eq:NestedInnerloop1}\\
\lambda_{k\ell}^{p,n+1} &$= \max\left\{0,\lambda_{k\ell}^{p,n} + \beta \phi_{k\ell}^p(\XX^{p,n+1}) \right\},\  k,\ell = 1,..., N,\ k< \ell,$
\label{eq:NestedInnerloop2}
\end{numcases}
then $(\XX^{p+1},\llambda^{p+1}) = \lim_{n\to \infty} (\XX^{p,n}, \llambda^{p,n})$.
\end{definition}

If we only compute one step of the inner-loop per iteration of the outer-loop we get a variant of the AHA formulation, where $\phi_{k\ell}(\XX^{p+1})$ is replaced by $\phi_{k\ell}^p(\XX^{p+1})$ in~\eqref{eq:NestedInnerloop2}.

\subsubsection{The nested algorithm for the velocities (NAV)}

We consider the minimization problem~\eqref{eq:systemSaddlePoint_V} with linearized constraint functions.

\begin{definition}[\textbf{Nested Algorithm for the Velocities (NAV)}] Let $\tau > 0$ and $(\XX^p, \VV^p, \mmu^p)$ be given.
Define $\VV^{p,0}= \VV^p$, $\mmu^{p,0}= \mmu^p$ and $\phi_{k\ell}^p$ as in~\eqref{eq:linearizationConstraints}.
For a given $(\VV^{p,n}, \mmu^{p,n})$, let the step of the inner-loop be defined as
\begin{numcases}{}
\hspace{-0.5cm}&$V^{p,n+1}_i       = V^{p,n}_i$\nonumber\\
\hspace{-0.5cm}& $ - \alpha \left(  V_i^{p,n} +\nabla_{X_i}W(\XX^p) +\tau\sum\limits_{k,\ell \in \{1,...,N\},\ k<\ell} \mmu_{k\ell}^{p,n} \nabla_{X_i} \phi_{k\ell}^p(\XX^p + \tau\VV^{p,n})\right),$ \nonumber\\
\hspace{-0.5cm} &\hspace{6cm}$\ i=1,...,N$ \label{eq:velInnerLoop1}\\
\hspace{-0.5cm}&$\mu^{p,n+1}_{k\ell}  =  \max\left\{0, \mu_{k\ell}^{p,n}+ \beta\phi_{k\ell}^p(\XX^p + \tau\VV^{p,n+1})\right\},k,\ell = 1,...,N,$\label{eq:velInnerLoop2}
\end{numcases}
then $(\VV^{p+1},\mmu^{p+1}) = \lim_{n\to \infty} (\VV^{p,n}, \mmu^{p,n})$ and $\XX^{p+1}         = \XX^p + \tau\ \VV^{p+1}$.
\end{definition}

The NAV corresponds to an adaptation of the  method developed by Maury in~\cite{Maury2006}.
If at each iteration of the outer-loop, we only perform one iteration of the inner-loop, we get a combined scheme, which does not seem to converge for $N$ large.

\subsection{Smooth and non-smooth form of the constraint functions}

The non-overlapping constraints for a system of identical spheres in $\RR^b$ can be expressed by means of a smooth or a non-smooth function as specified bellow, both leading to equivalent constraints.

\begin{definition}
We call \textbf{non-smooth form of the constraint functions (NS)} the following function
$$\phi_{k\ell}(\XX) = d - |X_k - X_{\ell}|,k,\ell=1,...,N,\ k\neq \ell$$ 
and \textbf{smooth form of the constraint functions (S)} the following function 
$$\phi_{k\ell}(\XX) = d^2 - |X_k - X_{\ell}|^2,k,\ell=1,...,N,\ k\neq \ell.$$ 
\end{definition}\label{def:S/NS}

\section{Linear analysis}
\label{sec:linearAnalysis}
\setcounter{equation}{0}

\subsection{Preliminaries}

In the following, we carry out the linear stability analysis of the associated ODE systems in order to study the convergence of the solution towards a steady state. We consider here the particular physical system where $N$ rigid spheres in $\RR^b$ attract each other through a global potential which is given by a quadratic function of the relative distance, 
\begin{eqnarray}
W(\XX) = \frac{1}{2N}\sum_{i,j \in \{1,...,N\},\ i<j} |X_i-X_j|^2. \label{eq:W}
\end{eqnarray}

\begin{definition}
A steady state $x^*$ of the ODE system $\dot x = f(x),\ t\geq 0,$ is called
\begin{itemize}
\item \textbf{stable} (in the sense of Lyapunov) if for all $ \epsilon >0$, there exists a $\delta > 0$ such that $\| \bar x(0) - x^* \| < \delta$ implies $\|\bar  x(t)-x^*\| < \epsilon,$ for all $t>0$ and for all solution $\bar x$;
\item \textbf{asymptotically stable} if it is stable and $\lim_{t\to \infty} \|\bar x(t) - x^* \| = 0$;
\item \textbf{unstable} if it is not stable.
\end{itemize}
\end{definition}

Note that this definition assumes that the initial configuration is chosen close enough to the steady state.
The next theorem allow us to obtain conclusions about the original nonlinear system from the corresponding linearized system.

\begin{theorem}\label{thm:linearStability}
Consider the ODE system $\dot x = f(x)$ and a steady state $x^*$, where $f$ is smooth at $x^*$.
If $x^*$ is an asymptotically stable (unstable) solution of the linearized system about $x^*$, i.e., $\dot{ \tilde x} = f'(x^*)(\tilde x - x^*)$, then it is an asymptotically stable (unstable) solution of the original system.
\end{theorem}

\begin{proof}
See~\cite{Chicone1999}, thm. 2.42, p. 158.
\end{proof}

In order to ensure convergence of the ODE system towards a steady state, we only need to ensure that the eigenvalues of $f'(x^*)$ all have negative real part. If at least one eigenvalue has positive real part, then $x^*$ is unstable, and if all eigenvalues are pure imaginary, then $x^*$ is a {\it center equilibrium}, i.e. if a solution starts near it then it will be periodic around it. In the latter case, we can not conclude anything about the nonlinear system.
The analysis presented next is made for the case of two spheres in $\RR$.

\subsection{The Arrow-Hurwicz algorithm (AHA)}\label{sec:DAHAanal}

\subsubsection{AHA-NS}\label{sec:UZnon-smooth}

Let $\phi(X) = d - |X|$ and consider the potential~\eqref{eq:W}. The ODE system associated to the DAHA-NS in the case of two spheres in $\RR$ where one sphere is fixed at the origin can be written as
\begin{numcases}{}
\dot X       & $=\ - \alpha \left(1 - \frac{\lambda}{|X|}\right)X $\label{eq:UzNon-smooth1Shpere1}\\
 \dot \lambda &$=
 \begin{cases}
 0,  \mbox{ if } \lambda =0 \text{ and } d < |X|\\
 \beta (d - |X|), \mbox{ otherwise}.
 \end{cases} $\label{eq:UzNon-smooth1Sphere2}
\end{numcases}

\begin{lemma}
The steady states of the system~\eqref{eq:UzNon-smooth1Shpere1}-\eqref{eq:UzNon-smooth1Sphere2}, $(X^*,\lambda^*) = (d , d)$ and $(X^*,\lambda^*) = (-d , d)$, are both asymptotically stable, for any $\alpha$ and $\beta$ positive.
\end{lemma}

\begin{proof}
Since the dynamics around each steady state is identical, we only need to carry out the analysis of the first steady state.
Suppose $X>0$ and consider the change of variables $Y=X-d$ and $\mu = \lambda - d$.
The system on the new variables is given in matrix form by
$$ \begin{bmatrix}
\dot Y\\
\dot \mu
\end{bmatrix}
=
A
\begin{bmatrix}
Y\\
\mu
\end{bmatrix},
\ 
A=
\begin{bmatrix}
-\alpha  &  \alpha\\
-\beta   & 0
\end{bmatrix}.
$$
We want the eigenvalues of matrix A to be real and negative in order to have a fast convergence to the steady state. 
The roots of the characteristic polynomial $\PP(\lambda) = \lambda^2 + \alpha\lambda + \beta \alpha,$ have both negative real part, therefore the steady state is asymptotically stable.
\end{proof}

Any solution to the ODE system~\eqref{eq:UzNon-smooth1Shpere1}-\eqref{eq:UzNon-smooth1Sphere2} converges to a steady state for all $\alpha, \beta>0$ and the fastest convergence is achieved when $\alpha = 4\beta$.
Contrarily to the one dimensional case, in higher spatial dimensions the constraints are no longer piecewise linear. Consequently, we can not directly extrapolate the conclusions drawn in this section.
In particular, in dimension $b=2$, the numerical simulations show oscillations around the steady state for $N>3$ without never converging to it.

\subsubsection{AHA-S}~\label{sec:UZsmooth}

Let $\phi(X) = d^2 - |X|^2$  and consider the potential~\eqref{eq:W}.
The ODE system associated to the AHA-S in the case of two spheres in $\RR$ where one sphere is fixed at the origin can be written as
\begin{numcases}{}
\dot X       &$= - \alpha \left(1 - 2\lambda\right)X$\label{eq:nonlinearSmooth1} \\
\dot \lambda &$=
\begin{cases}
0, \text{ if }\lambda = 0 \mbox{ and } d < |X|\\
\beta (d^2 - X^2),  \mbox{ otherwise.} \\
 \end{cases}$\label{eq:nonlinearSmooth2}
\end{numcases}

\begin{lemma} The steady states of the system corresponding to the linearization of~\eqref{eq:nonlinearSmooth1}-\eqref{eq:nonlinearSmooth2}, $(X^*,\lambda^*) = (d , 1/2)$ and $(X^*,\lambda^*) = (-d , 1/2)$, are both center equilibria, for any $\alpha$ and $\beta$ positive. 
\end{lemma}

\begin{proof}
As before, we will only carry out the analysis of the first steady state.

Suppose $X>0$ and consider the change of variables $Y=X-d$ and $\mu = \lambda - 1/2$.
The linearized system on the new variables is given in matrix form by
$$ \begin{bmatrix}
\dot Y\\
\dot \mu
\end{bmatrix}
=
A
\begin{bmatrix}
Y\\
\mu
\end{bmatrix},
\ 
A=
\begin{bmatrix}
0          & 2d\alpha \\
-2d\beta   & 0
\end{bmatrix}.
$$
The roots of the characteristic polynomial $\PP(\lambda) = \lambda^2 + 4d^2\alpha\beta$ are both purely imaginary, therefore the steady state of the linearized system is a center equilibrium.
\end{proof}
\begin{figure}[!h]
\centering
\includegraphics[scale=0.6]{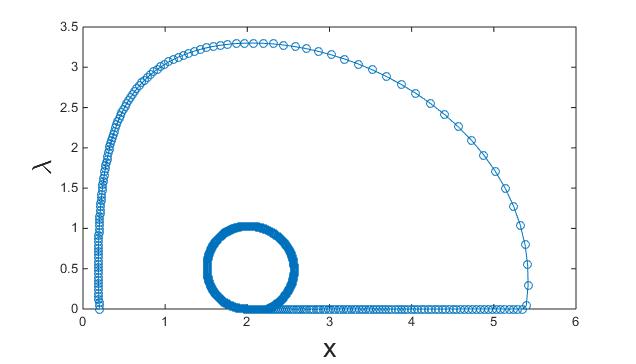}
\caption{Phase portrait of the system~\eqref{eq:nonlinearSmooth1}-\eqref{eq:nonlinearSmooth2} with $(\alpha, \beta, d) = (0.01,0.01,2) $ and initial condition $X^0 = 0.2$. The dynamics do not converge to the equilibrium $(2,\frac12)$.}
\label{fig:phasePortraitAH}
\end{figure}

The linear analyses does not allow us to conclude anything about the asymptotic behaviour of the nonlinear system (see thm.~\ref{thm:linearStability}).
Nevertheless, the phase portrait plotted in Figure~\ref{fig:phasePortraitAH} reveals that a solution to the nonlinear system should converge towards a periodic orbit around the steady state.
As we will see in the next section, the damping term applied to the Arrow-Hurwicz system~\eqref{eq:AHSystem1}-\eqref{eq:AHSystem2} ensures asymptotic stability of the steady state, under certain conditions on the parameters.

\subsection{The damped Arrow-Hurwicz algorithm (DAHA)}\label{sec:DAHAanal}

\subsubsection{DAHA-NS}\label{sec:DAHSnon-smooth}

Let $\phi(X) = d - |X|$ and consider the potential~\eqref{eq:W}. 
The ODE system associated to the DAHA-NS in the case of two spheres in $\RR$ where one sphere is fixed at the origin can be written as
\begin{numcases}{}
\ddot{ X }      &$= -\alpha^2 \left(1- \frac{\lambda}{|X|} \right)X
                   +\alpha\beta \lambda(d - |X|)\frac{X}{|X|}
                   -c\dot X $\label{eq:DAHSnon-smooth1Sphere1}\\
\dot \lambda &$= 
\begin{cases}
0, \text{ if }\lambda = 0 \mbox{ and } d < |X|,\\
\beta (d - |X|), \mbox{ otherwise}.
 \end{cases}$\label{eq:DAHSnon-smooth1Sphere2}
\end{numcases}

\begin{lemma}
Let $ \alpha, \beta, c>0$. If $(\alpha+\beta d) c -\beta \alpha > 0$, then the steady states of the system~\eqref{eq:DAHSnon-smooth1Sphere1}-\eqref{eq:DAHSnon-smooth1Sphere2}, $(X^*,\dot{X}^*,\lambda^*)=(d,0,d)$ and $(X^*,\dot{X}^*,\lambda^*) = (-d,0,d)$, are both asymptotically stable.
\end{lemma}

\begin{proof}
Suppose $X>0$ and consider the change of variables $Y=X-d$, $Z=\dot Y$ and $\mu = \lambda - d$.
The linearized system on the new variables is given in matrix form by
$$ \begin{bmatrix}
\dot Y\\
\dot Z\\
\dot \mu
\end{bmatrix}
=
A
\begin{bmatrix}
Y\\
Z\\
\mu
\end{bmatrix},
\ \ \ A=
\begin{bmatrix}
0                        &   1  &   0\\
-\alpha^2-\alpha\beta d  &  -c  &  \alpha^2\\
-\beta                   &  0   & 0
\end{bmatrix}.
$$  
The eigenvalues of matrix A are the roots of the characteristic polynomial in  $\lambda$, which is given by
$\PP(\lambda) = \lambda^3 + c\lambda^2 + (\alpha^2+\alpha\beta d)\lambda + \beta \alpha^2.$
Consider in general a cubic polynomial of the form
$\PP(\lambda) = \lambda^3 + c_2\lambda^2 + c_1\lambda + c_0,$
with $c_0,c_1,c_2 \in \RR^+$.
Let $z_1, z_2$ and $z_3$ be the (complex) roots to this polynomial.
We want to ensure that all roots have negative real part.
Since all coefficients are positive, if the roots are real then they must be negative.
Suppose now that two roots are complex conjugate, for example, $z_1 = a+ib,\ z_2 = a-ib$, $a,b \in \RR$ and $z_3 \in \RR^-$.
In order to find a condition on the coefficients which ensures that $a$ is non-positive, we start by identifying the coefficients of the equation with its roots:
$$
z_1+z_2+z_3          = -c_2,\ 
z_1z_2+z_1z_3+z_2z_3 = c_1,\ 
z_1z_2z_3            = -c_0
$$

Rewriting in terms of $a,b$ and $z_3$ we get
\begin{eqnarray}
2a+z_3          = -c_2,\ 
a^2+b^2+2az_3   = c_1,\ 
(a^2+b^2)z_3    = -c_0\label{eq:roots}
\end{eqnarray}
From \eqref{eq:roots} we deduce that $a$ satisfies the cubic polynomial
$$8a^3+8c_1a^2+2(c_1+c_2^2)a+c_1c_2-c_0=0.$$
Consequently, if $c_1c_2-c_0>0$, then $a$ is necessarily negative.

Back to our case, we have
$c_2          = c,\ c_1          = \alpha^2+\alpha\beta d$ and $c_0          = \beta\alpha^2$
and a sufficient condition for the steady state to be asymptotically stable is
$(\alpha^2+\alpha\beta d)c-\beta\alpha^2>0$, i.e., $(\alpha+\beta d)c-\beta\alpha>0$.
Note that since the steady state is asymptotically stable as a solution to the linearized system, then it is also asymptotically stable (see thm.~\ref{thm:linearStability}).
\end{proof}

\subsubsection{DAHA-S}\label{sec:DAHSsmooth}

Let $\phi(X) = d^2 - |X|^2$ and consider the potential~\eqref{eq:W}.
For the case of two spheres in $\RR$ where one sphere is fixed at the origin, the ODE system associated to the DAHA-S can be written as
\begin{numcases}{}
\ddot{ X }      &$= -\alpha^2 \left(1- 2\lambda \right)X
                  +2\alpha\beta \lambda(d^2 - |X|^2)X
                   -c\dot X$\label{eq:DAHSsmooth1Sphere1}\\
\dot \lambda  &$=
 \begin{cases}
  0, \text{ if }\lambda       =0 \mbox{ and } d < |X|\\
 \beta (d^2 - |X|^2), \mbox{ otherwise}.
 \end{cases}$\label{eq:DAHSsmooth1Sphere2}
\end{numcases}

\begin{lemma}
Let $\alpha, \beta, c>0$. If $c-2\alpha>0$, then the steady states of the system~\eqref{eq:DAHSsmooth1Sphere1}-\eqref{eq:DAHSsmooth1Sphere2}, $(X^*,\dot{X}^*,\lambda^*)=(d,0,1/2)$ and $(X^*,\dot{X}^*,\lambda^*) = (-d,0,1/2)$, are both asymptotically stable.
\end{lemma}

\begin{proof}
As before, suppose $X>0$ and consider the change of variables $Y=X-d$, $Z=\dot Y$ and $\mu = \lambda - 1/2$.
The linearized system on the new variables is given in matrix form by

$$ \begin{bmatrix}
\dot Y\\
\dot Z\\
\dot \mu
\end{bmatrix}
=
A
\begin{bmatrix}
Y\\
Z\\
\mu
\end{bmatrix},
\ \ 
A=
\begin{bmatrix}
0                        &   1  &   0\\
-2\alpha\beta d^2        &  -c  &  2d\alpha^2\\
-2d\beta                 &  0   & 0
\end{bmatrix}.
$$

The eigenvalues of matrix A are the roots of the characteristic polynomial in $\lambda$:
$$\PP(\lambda) = \lambda^3 + c\lambda^2 + 2\alpha\beta d^2\lambda + 4d^2\beta \alpha^2$$

Using the same reasoning as before we have $c_2= c,\ c_1          = 2\alpha\beta d^2$ and $c_0= 4d^2\beta\alpha^2$.
A sufficient condition for the steady state to be asymptotically stable is
$2c\alpha\beta d^2-4d^2\beta\alpha^2>0$, i.e., $c-2\alpha>0$.
\end{proof}
\begin{remark}
We see that as long as the damping coefficient, $c$, is large enough, the sufficient conditions for stability of both the DAHA-NS and DAHA-S are fulfilled. 
Furthermore, the parameter space corresponding to the stability of DAHA-NS is larger than the one of the DAHA-S. 
\end{remark}

\subsection{Previous approaches}

\subsubsection{NAP-NS}

Let $\phi(X) = d-|X|$ and consider its linearization around $X^p\neq 0$, i.e., $\phi^p(X) = d-\frac{X^p}{|X^p|}X$, and the potential~\eqref{eq:W}.
The ODE system associated to the inner-loop of the NAP-NS in the case of two spheres in $\RR$ where one sphere is fixed at the origin can be written as
\begin{numcases}{}
 \dot X       &$= - \alpha \left(X -  \frac{X^p}{|X^p|}\lambda \right)$ \label{eq:NestedNONSmooth1Sphere1}\\
\dot \lambda  &$=
\begin{cases}
0, \text{ if } \lambda =0 \mbox{ and } d - \frac{X^p}{|X^p|}X < 0\\
\beta (d - \frac{X^p}{|X^p|}X), \mbox{ otherwise},
\end{cases}$\label{eq:NestedNONSmooth1Sphere2}
\end{numcases}
with the initial condition $(X,\lambda)(0) = (X^p,\lambda^p)$.

\begin{lemma} 
If $X^0~\neq~0$, then the steady state  of the system~\eqref{eq:NestedNONSmooth1Sphere1}-\eqref{eq:NestedNONSmooth1Sphere2}, $(X^*,\lambda^*)=d\left(\frac{|X^p|}{X^p},1\right)$, is asymptotically stable for any $\alpha$ and $\beta$ positive and the outer-loop converges in one iteration.
\end{lemma}
\begin{proof}
The stability analysis shows that the steady state of~\eqref{eq:NestedNONSmooth1Sphere1}-\eqref{eq:NestedNONSmooth1Sphere2}, namely, $(X^*,\lambda^*)=d\left(\frac{|X^p|}{X^p},1\right)$, is asymptotically stable. 
Furthermore, if $X_0 \neq 0$, the outer-loop is defined recursively by $X^{p+1} = d \frac{|X^p|}{X^p}$ and it converges in one iteration. In fact, $X^1 = d \frac{|X^0|}{X^0}$ and $X_p = X_1$ for all $p>1$.
\end{proof}
 
The conclusions of the analysis in the one dimensional case can not be directly extrapolated to higher dimensional cases, where the constraint functions are no more piecewise linear. 
In section~\ref{sec:numerics}, we resort to numerical simulations to get some insight about the behaviour of the system in two spatial dimensions.

\subsubsection{NAP-S}

Let $\phi(X) = d^2 - | X|^2$ and consider its linearization around $X^p$, i.e., $\phi^p(X) = d^2+|X^p|^2-2 X^p\cdot X$ and the potential~\eqref{eq:W}.
The ODE system associated to the inner-loop of the NAP-S in the case of two spheres in $\RR$ where one sphere is fixed at the origin can be written as
\begin{numcases}{}
 \dot X       &$= - \alpha \left(X -  2X^p\lambda \right)$ \label{eq:NestedSmooth1Sphere1}\\
\dot \lambda  &$=
\begin{cases}
0, \text{ if } \lambda =0 \mbox{ and } d^2 + (X^p)^2 - 2X^pX < 0\\
\beta (d^2 +(X^p)^2- 2 X^p X), \mbox{ otherwise},
\end{cases}$\label{eq:NestedSmooth1Sphere2}
\end{numcases}
with the initial condition $(X,\lambda)(0)= (X^p,\lambda^p)$.

\begin{lemma}
If $X^0~\neq~0$, then the steady state of the system~\eqref{eq:NestedSmooth1Sphere1}-\eqref{eq:NestedSmooth1Sphere2}, $(X^*,\lambda^*) = \frac{d^2+ (X^p)^2}{2(X^p)^2}(X^p , \frac{1}{2})$, is asymptotically stable for any $\alpha$ and $\beta$ positive and the outer-loop generates the sequence $\{X^p\}_{p\in \NN}$ defined iteratively by $X^{p+1} = \frac{d^2 + (X^p)^2}{2X^p}$, which is convergent.
\end{lemma}

\begin{proof}
The stability analysis shows that the steady state of~\eqref{eq:NestedSmooth1Sphere1}-\eqref{eq:NestedSmooth1Sphere2}, namely, $(X^*,\lambda^*) = \frac{d^2+ (X^p)^2}{2(X^p)^2}(X^p , \frac{1}{2})$, is asymptotically stable.
Consequently, the outer-loop generates the sequence defined recursively by
$X^{p+1} =  \frac{d^2 + (X^p)^2}{2X^p},$
which is well-defined for $X^0 \neq 0$. 
If this sequence is convergent to, say, $L$, then $L$ must satisfy $L =  \frac{d^2 + L^2}{2L}$ i.e., $ L = \pm d$. 
Now, if $X^p > d$ then 
$$\frac{X^{p+1}}{X^p} = \frac{1}{2} \left( \frac{d^2}{(X^p)^2} + 1 \right)< 1,$$
therefore $X^{p+1} < X^p$, i.e., the sequence decreases. 
Furthermore, if we write $X^p$ in the form $X^p = d + \epsilon $, $\epsilon >0$, we then have that
\begin{eqnarray}
X^{p+1} - d &=& \frac{1}{2}\left( \frac{d^2}{d+\epsilon}+d+\epsilon\right) - d \nonumber \\
        &=& \frac{1}{2(d+\epsilon)}(d^2 + (d+\epsilon)^2 - 2d(d+\epsilon)) \nonumber\\
        &=&  \frac{\epsilon^2}{2(d+\epsilon)} > 0\nonumber
\end{eqnarray}
i.e., $X^{p+1} > d$.
On the other hand, if $0< X^p < d$ then the sequence increases. 
Consequently, we finally conclude that if $X^0>0$ then the sequence $\{X^p\}_{p\in \NN}$ converges towards~$d$.
Using the same reasoning we conclude that if $X^0<0$ then $\{X^p\}_{p\in \NN}$ converges towards~$-d$.
We see that for the case of the smooth form of the constraint functions, the sequence generated by the outer-loop converges.
\end{proof}

\subsubsection{NAV-NS}

Let $\phi(X) = d - |X|$ and consider the linearization of $\phi$ around $X^p~\neq~0$ evaluated at $X^p + \tau V$, i.e., $\phi^p(X^p + \tau V) = d- |X^p|-\tau\frac{X^p}{|X^p|} V$ and the potential~\eqref{eq:W}.
In the particular case of two spheres in $\RR$ where one sphere is fixed at the origin, the ODE system corresponding to the inner-loop is given by
\begin{numcases}{}
\dot V &$=  -\alpha \left( V + X^p -   \tau\ \lambda \frac{X^p}{|X^p|}\right)$ \label{eq:NestedVelocities1Sphere1} \\
\dot \lambda  &$=
\begin{cases}
0, \text{ if } \lambda =0 \mbox{ and } d-|X^p| - \tau \frac{X^p}{|X^p|}V<0\\
\beta\left(d-|X^p| - \tau \frac{X^p}{|X^p|}V\right), \mbox{ otherwise},
 \end{cases} $\label{eq:NestedVelocities1Sphere2}
\end{numcases}
with the initial condition $(V,\lambda)(0) = (V^p, \lambda^p)$.

\begin{lemma}
If $X^0~\neq~0$, then the the steady state of the system~\eqref{eq:NestedVelocities1Sphere1}-\eqref{eq:NestedNONSmooth1Sphere2}
is asymptotically stable for any $\alpha$ and $\beta$ positive and the outer-loop converges in one iteration.
\end{lemma}
\begin{proof}
The stability analysis shows that the steady state,
$$V^{*} = \frac{d-|X^p|}{\tau} \frac{|X^p|}{X^p},\ \lambda^* = \frac{1}{\tau^2}(d-|X^p|(1-\tau) ),$$
is asymptotically stable. 
Consequently, the outer-loop generates the sequence defined recursively by $X^{p+1} = d \frac{|X^p|}{X^p}$. %= X^p + \tau\ V^{p+1}  = X^p + \tau \frac{d-|X^p|}{\tau} \frac{|X^p|}{X^p}. 
If $X_0 \neq 0$ the sequence is well-defined and $X^1 = d \frac{|X^0|}{X^0}$ and $X_p = X_1$ for all $p>1$, hence the outer-loop converges in one iteration.
\end{proof}

The NAV with the smooth form of the constraint functions did not show numerically good convergence results, for we do not explore it in this paper. 
We note that in~\cite{Maury2006}, the author has also only considered the non-smooth form of the constraint functions.

\section{Numerical results}\label{sec:numerics}
\setcounter{equation}{0}

In this section we investigate and compare the numerical results obtained from the damped Arrow-Hurwicz algorithms (DAHA-NS, DAHA-S) and the nested algorithms (NAP-NS, NAP-S and NAV-NS) for the potential defined in~\eqref{eq:W} in two spatial dimensions, $b=2$.
In particular, we address the convergence time, the robustness of the convergence time with respect to the initial configurations and the accuracy of each method for the case of $N=7$. 
We briefly study the case $N=100$ and show an example of intermediate steps obtained with the DAHA-S for $N=2000$.
Note that for most values of $N$, the local minimum of $W$ that is found by each algorithm, $W(\bXX)$, depends on the initial configuration, except in very specific cases, such as $N=7$, where the minimum is unique, precisely $W(\bXX)=3$.
This fact motivates our choice of $N=7$ for studying the accuracy.

In order to adjust the spatial dimensions, the numerical parameters must satisfy
$\alpha,\beta, c \sim \OO(1)$ for the methods with the non-smooth form of the constraint functions and $\alpha, c \sim \OO(1)$ and $\beta \sim \OO(1/d^2)$ for the methods with the smooth form of the constraint functions. In the following we have considered $d=1$.

In order to be able to compare the nested algorithms with the DAHA regarding convergence time, we only consider the evolution of $\XX$ and $\VV$ and we do not consider the evolution of $\llambda$.
We denote by $\|\cdot \|$ the Euclidean norm in $\RR^{bN}$.
For a given small and positive $\epsilon$, the stopping criterion for the minimization algorithms associated to the NAP, is given by the following condition on the relative error
\begin{eqnarray}
\frac{\| \XX^{n+1} - \XX ^n\|}{\| \XX^n\|} &< \epsilon_{inner}.
\label{eq:convTest1}
\end{eqnarray}
For the case of the minimization problem formulated in terms of the velocities, the stopping criterion is similar but instead of $\XX$ we write $\VV$ and instead of normalizing by $\VV^n$, we normalize by $\XX^n$, yielding 
\begin{eqnarray}
 \frac{\| \VV^{n+1} - \VV ^n\|}{\|\XX^n\|} &< \frac{\epsilon_{inner}}{\tau}.
\label{eq:convTest2}
\end{eqnarray}
By using the Euler step 
$\XX^{n+1} = \XX^p + \tau\ \VV^{n+1}$
we show that the two conditions~\eqref{eq:convTest1} and~\eqref{eq:convTest2} are equivalent.
As we will see, in order to get a fast convergence with the nested algorithms, one does not need to wait for the convergence of the inner-loop. 
We introduce a new parameter, $I_{inner}$, which stands for the maximum number of iterations of the inner-loop allowed per outer-loop iteration.
Finally, the stopping criterion for both the outer-loop of the NAP and the NAV, as well as, for the DAHA reads
\begin{eqnarray}
 \frac{\| \XX^{p+1} - \XX^p\|}{\| \XX^p\|} < \epsilon.
\label{eq:conTest3}
\end{eqnarray}

The assessment and comparison of the methods
will be made through the comparison of statistical indicators obtained from averaging certain quantities over a set of different initial configurations. 
These indicators are introduced bellow.

\begin{definition}
Consider a set of $p$ initial configurations for which an algorithm converges, i.e., the stopping criterion is satisfied in a finite number of iterations.
Let $T_{\ell}$ be the number of iterations needed for the algorithm to converge when starting with the $\ell th$ initial configuration. Let $A_{ij}$ be the overlapping area of spheres $i$ and $j$ at convergence and $A_{total}=N\pi (d/2)^2$. 

We define the following statistical indicators
{\bf mean convergence time},
{\bf variance of the convergence time}
 and the 
{\bf mean proportion of overlapping area per sphere} as 

$$T=\frac{1}{p}\sum\limits_{\ell=1}^p\limits T_{\ell},
\ \ \ \sigma^2 = \frac{1}{p-1}\sum\limits_{\ell=1}^p\limits (T_{\ell} - T)^2 \ \ \text{and}\ \ \  A = \frac{1}{pN A_{total}}\sum\limits_{i,j \in \{1,...,N\},\ i<j} A_{ij},$$
respectively.
\end{definition}

The indicator $T$ measures the efficiency of an algorithm with respect to the convergence time, $A$ and $W$ measure the accuracy of the final configuration and $\sigma^2$ measures the robustness of the convergence time with respect to the initial configurations.
For simplicity we assume that the time interval between iterations is constant and invariant among the different algorithms. 
As a consequence of this simplification, we will use the number of iterations as the time unit of $T$.

\subsection{Case $N=7$}

We present a detailed numerical study for the case of $N=7$ spheres in dimension $b=2$.
The $20$ different initial configurations considered in this section were generated from a standard Gaussian distribution.
%The results described in this section were obtained by considering $20$ randomly generated initial configurations from a standard Gaussian distribution. 
We choose the tolerances $\epsilon=10^{-6}$ and $\epsilon_{inner} = 10^{-9}$ and the maximum number of iterations of the inner-loop $I_{inner} = 10$.
\begin{figure}
\begin{subfigure}{.5\textwidth}
  \centering
  \includegraphics[width=1\linewidth]{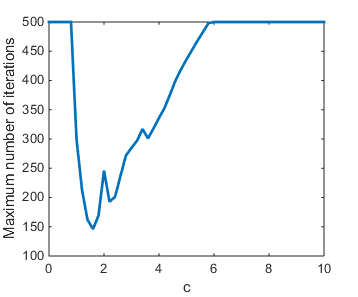}
  \caption{DAHA-NS}
  \label{fig:DAHANS_N=7_diffc}
\end{subfigure}
\begin{subfigure}{.5\textwidth}
  \centering
  \includegraphics[width=1\linewidth]{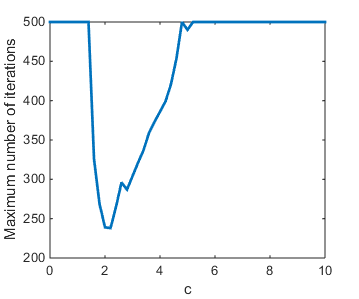}
  \caption{DAHA-S}
  \label{fig:DAHAS_N=7_diffc}
\end{subfigure} 
\caption{Maximum number of iterations needed for the DAHA to converge over a set of $20$ randomly generated initial configurations as a function of $c$ for $N=7$ and $\epsilon=10^{-6}$.
The numerical parameters used are:  
~\ref{fig:DAHANS_N=7_diffc} $(\alpha,\beta)=(0.3,3)$ and
~\ref{fig:DAHAS_N=7_diffc} $(\alpha,\beta)=(0.3,1.4)$.
}
\label{fig:DAHA_N=7_diffc}
\end{figure} 
\begin{figure}
\begin{subfigure}{.5\textwidth}
  \centering
  \includegraphics[width=1.05\linewidth]{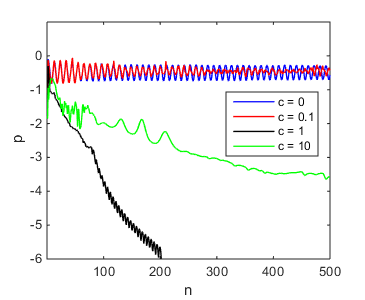}
  \caption{{DAHA-NS}}
  \label{fig:DAHANS_N=7_diffc_examples}
\end{subfigure}
\begin{subfigure}{.5\textwidth}
  \centering
  \includegraphics[width=1.05\linewidth]{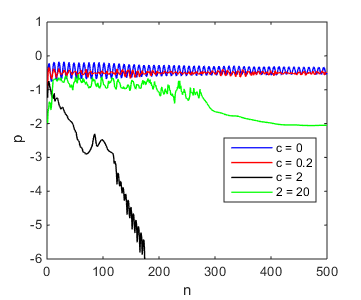}
  \caption{DAHA-S}
  \label{fig:DAHAS_N=7_diffc_examples}
\end{subfigure} 
\caption{Relative error on log scale averaged over a set of $20$ randomly generated initial configurations, $p$, as a function of iteration number, $n$, for different values of $c$ and for $N=7$ and $\epsilon=10^{-6}$.
The numerical parameters used are:  
~\ref{fig:DAHANS_N=7_diffc_examples} $(\alpha,\beta)=(0.3,3)$ and
~\ref{fig:DAHAS_N=7_diffc_examples} $(\alpha,\beta)=(0.35,1.4)$.
}
\label{fig:DAHA_N=7_diffc_examples}
\end{figure} 
In order to study the relation between the damping parameter $c$ and the convergence time of the DAHA with smooth and non-smooth constraints, we plot
in Figure~\ref{fig:DAHA_N=7_diffc} the maximum number of iterations over $20$ different randomly generated initial configurations as a function of $c \in (0, 10]$.
We observe that the lower convergence time is attained when $c \approx 2$, for both the DAHA with the smooth and with the non-smooth constraints.
In Figure~\ref{fig:DAHA_N=7_diffc_examples} we plot the relative error as a function of iteration number, $n$, for different values of $c$. 
If $c=0$ we observe that the relative error oscillates and never drops bellow $10^{-1}$.  
As we increase $c$ the oscillations tend to diminish. 
In the following we have used $c=2$.
Note that this choice for $c$ eliminates the dependence on $\XX^{n-1}$ in~\eqref{eq:secc_discrete1}-\eqref{eq:secc_discrete2}, in this case, the DAHA can be seen as a discretization of the following first-order ODE system:
\begin{numcases}{}
 \dot X_i &$= - \frac{1}{2}\alpha^2 [\nabla_{X_i}  W (\XX) + \sum\limits_{k,\ell \in \{1,...,N\},\ k<\ell} \lambda_{k\ell} \nabla_{X_i} \phi_{k\ell}(\XX)]$ \nonumber\\ 
& $ -\frac{1}{2} \alpha \beta \sum\limits_{k,\ell \in \{1,...,N\},\ k<\ell} \phi_{k\ell}(\XX)  \lambda_{k\ell} \nabla_{X_i} \phi_{k\ell}(\XX), \ i = 1,...,N$ \nonumber\\
\dot \lambda_{k\ell}  &$=\ 
\begin{cases}
0,\ \text{ if } \lambda_{k\ell}=0 \text{ and } \phi_{k\ell}(\XX)<0\\
\beta \phi_{k\ell}(\XX),\ \text{ otherwise }
\end{cases}$,\nonumber \\
& \hspace{4cm}$ k,\ell = 1,...,N,\   k< \ell$.  \nonumber
\end{numcases}
\begin{figure}
\begin{subfigure}{.5\textwidth}
  \centering
  \includegraphics[width=0.75\linewidth]{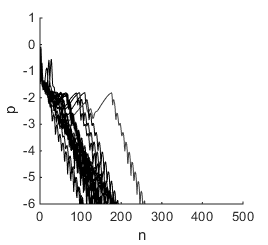}
  \caption{{DAHA-NS}}
  \label{fig:itBest_N=7_1}
\end{subfigure}
\begin{subfigure}{.5\textwidth}
  \centering
  \includegraphics[width=0.75\linewidth]{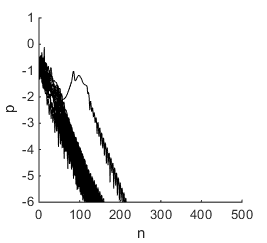}
  \caption{DAHA-S}
  \label{fig:itBest_N=7_2}
\end{subfigure} \\
\begin{subfigure}{.5\textwidth}
  \centering
  \includegraphics[width=0.75\linewidth]{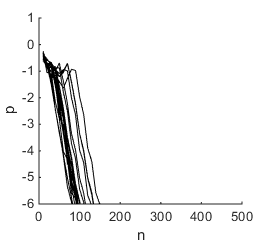}
  \caption{NAP-NS}
    \label{fig:itBest_N=7_3}
\end{subfigure}%
\begin{subfigure}{.5\textwidth}
  \centering
  \includegraphics[width=0.75\linewidth]{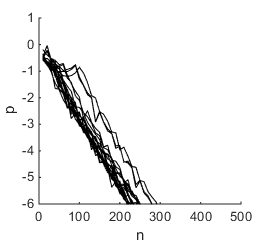}
  \caption{NAP-S}
    \label{fig:itBest_N=7_4}
\end{subfigure} \\
\begin{subfigure}{.5\textwidth}
  \centering
  \includegraphics[width=0.75\linewidth]{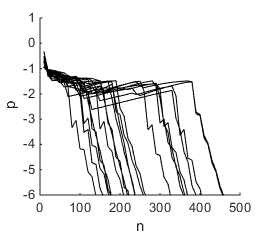}
  \caption{NAV-NS}
    \label{fig:itBest_N=7_5}
\end{subfigure}%
\caption{Relative error on log scale, $p$, as a function of iteration number, $n$, for $20$ randomly generated initial configurations computed by each algorithm for $N=7$ and $\epsilon=10^{-6}$. 
The numerical parameters used are:  
~\ref{fig:itBest_N=7_1} $(\alpha,\beta,c)=(0.3,3,2)$,
~\ref{fig:itBest_N=7_2} $(\alpha,\beta,c)=(0.35,1.4,2)$,
~\ref{fig:itBest_N=7_3} $(\alpha,\beta,I_{inner}, \epsilon_{inner})=(0.6,0.46,10,10^{-9})$,
~\ref{fig:itBest_N=7_4} $(\alpha,\beta,I_{inner}, \epsilon_{inner})=(0.25,0.28,10,10^{-9})$  and
~\ref{fig:itBest_N=7_5} $(\alpha,\beta,I_{inner}, \epsilon_{inner},\tau)=(0.48,126,10,10^{-9},0.1)$ .}
\label{fig:itBEST_N=7} 
\end{figure}
We now consider the five methods, namely, the DAHA-NS, DAHA-S, NAP-NS, NAP-S and NAV-NS. 
The numerical simulations suggest that $\alpha$ has influence in the attraction and $\beta$ in the repulsion between spheres, which can also be observed by looking directly at the equations. 
If $\beta$ is too large we observe oscillations, if $\alpha$ or both parameters are too large we observe numerical instability and if $\alpha$ is small we observe a very slow dynamics. 
The optimal set of parameters should then be chosen near the parameters that lead to numerical instability, which may be a problem, since an algorithm may be very efficient for some set of initial configurations, however it may be numerical unstable for another. 
Moreover, we have observed that the final positions obtained from each method by varying the numerical parameters are naturally not always the same. 
Only the relative positions (apart from permutations) are invariant. As an exception, the NAV leads to nearly exactly the same configurations. 
In fact, we can see from equations~\eqref{eq:vel1}-\eqref{eq:vel2} that the dynamics of the particles in the NAV is not directly affected by a change of the numerical parameters, since the parameters are only involved in the computation of the velocity. 
Contrarily, in the NAP and DAHA a change in the parameters produces a different dynamics, which leads to different final configurations. 
This is comprehensible, since the parameters in those algorithms are directly related to the attraction and repulsion forces between the spheres. 
The previous observation could be statistically verified by comparing all the final configurations produced by each method and checking how different (how far away from each other) they are. 

In the following, we have chosen the numerical parameters $(\alpha, \beta)$ that correspond to a fast convergence of each method for all the $20$ sets of initial configurations.
We plot in Figure~\ref{fig:itBEST_N=7} the relative error on log scale as a function of iteration number for $20$ randomly generated initial configurations and for each method.
We observe that the profile of the relative error follows the pattern: non-monotone behaviour, followed by an approximately linear decay at a certain speed, which seems to be invariant with respect to the initial configuration. 
The faster decay is observed in the NAP-NS (see Figure~\ref{fig:itBest_N=7_3}).  
The efficiency of the NAV-NS (see Figure~\ref{fig:itBest_N=7_5}) is apparently highly dependent on the initial configuration.
\begin{table}
\centering
\begin{tabular}{c|c|c|c}
\hline 
  $(\epsilon, \epsilon_{inner})$ & $(10^{-4},10^{-7})$& $(10^{-6},10^{-9})$  & $(10^{-8},10^{-11})$ \\ \hline
\multicolumn{4}{ c }{Mean convergence time ($T$)} \\ \hline
DAHA-NS                         &   $103$  &  $145$  & $198$\\ 
DAHA-S                          &  $85$  &   $133$  &   $182$  \\       
{\bf NAP-NS}                          &  ${\bf 81}$  &  ${\bf 107}$  &  ${\bf 131}$\\ 
NAP-S                          &  $165$  &  $246$  &  $325$ \\ 
NAV-NS                          &    $248$     & $283$      &  $535$    \\ \hline
\multicolumn{4}{ c }{Order of accuracy $(A, W-3)$} \\ \hline
DAHA-NS                     & $(10^{-6},-10^{-6})$ & $(10^{-9},-10^{-9})$ & $(10^{-12},-10^{-13})$\\ 
DAHA-S                      & $(10^{-7}, -10^{-6})$ &$(10^{-10},-10^{-10})$& $(10^{-12},-10^{-13})$ \\
{\bf NAP-NS} 						& ${\bf(10^{-9},10^{-8})}$ &  ${\bf(10^{-12},10^{-12})}$ &  ${\bf(10^{-15},0)}$ \\
NAP-S 						& $(10^{-8}, -10^{-5})$ & $( 10^{-11}, 10^{-6})$ & $(10^{-13},-10^{-8})$ \\
NAV-NS 						& $(10^{-8},10^{-1})$ & $( 10^{-10},-10^{-1})$	 & $(10^{-14},10^{-9})$ \\ \hline
\multicolumn{4}{ c }{Variance of the convergence time ($\sigma^2$)} \\ \hline
DAHA-NS                         & $1.08 \times 10^3 $ & $1.21 \times 10^3 $  & $1.70 \times 10^3 $ \\
DAHA-S							& $4.19 \times 10^2 $ & $4.49 \times 10^2$ & $ 5.65 \times 10^2 $ \\
{\bf NAP-NS}  						& ${\bf 3.00 \times 10^2}$ & ${\bf 4.13 \times 10^2} $  & ${\bf 2.93 \times 10^2 }$ \\
NAP-S 							& $5.74 \times 10^2 $ & $4.68 \times 10^2 $ & $ 4.68 \times 10^2 $ \\
NAV-NS 							& $ 1.03 \times 10^4$ & $ 1.04 \times 10^4$ & $1.80 \times 10^5 $ \\ \hline
\end{tabular}
\caption{Results of the assessment of the final configurations averaged over a set of $20$ initial configurations and obtained by each algorithm for $N=7$ and for three different values of $\epsilon$, namely, $10^{-4},\ 10^{-6}$ and $10^{-8}$. The parameters used are 
DAHA-NS $(\alpha,\beta,c)=(0.3,3,2)$,
DAHA-S $(\alpha,\beta,c)=(0.35,1.4,2)$,
NAP-NS $(\alpha,\beta,I_{inner}, \epsilon_{inner})=(0.6,0.46,10,10^{-9})$,
NAP-S $(\alpha,\beta,I_{inner}, \epsilon_{inner})=(0.25,0.28,10,10^{-9})$  and
NAV-NS $(\alpha,\beta,I_{inner}, \epsilon_{inner},\tau)=(0.48,126,10,10^{-9},0.1)$ .}
\label{tab:assessmentSol_N=7}
\end{table}
In order to quantify and compare the efficiency, as well as, the accuracy of the final configurations generated by each algorithm we consider three different tolerances, $\epsilon = 10^{-4}$, $\epsilon = 10^{-6}$ and $\epsilon = 10^{-8}$, and we compute the mean convergence time, $T$, the variance of the convergence time, $\sigma^2$, the mean proportion of overlapping area per sphere $A$ and the difference between the theoretical optimum and the value of $W$ at convergence. 
The results are presented in Table~\ref{tab:assessmentSol_N=7} and they are averaged over a set of $20$ randomly generated initial configurations. 
We observe that the NAP-NS (in bold) performs better than any other method, while the NAV-NS is the least robust to initial configurations, the slowest to reach convergence and it only produces an accurate solution for $\epsilon = 10^{-8}$.
In the case of the smooth constraints, we observe that the DAHA-S seems to converge faster and produce more accurate solutions than the NAP-S.

\subsection{Case $N=100$}

We present in the following a short study for the case of $N=100$ spheres. 
The $5$ different initial configurations considered in this section were generated from a standard Gaussian distribution. 
In this case where the initial configurations are very dense, one may observe two different types of behaviours depending on the choice of the numerical parameters: either the spheres disperse initially very rapidly before they start to concentrate again  while trying to avoid overlapping with other spheres or they disperse slowly while trying to rearrange in a non-overlapping configuration.
We keep the choice $\epsilon=10^{-6}$ and $\epsilon_{inner}=10^{-9}$, and we choose $I_{inner} = 10^3$.
Similarly to the case $N=7$, the best value for the damping parameter $c$ should be of order $\OO(1)$. We keep the choice $c=2$.

\begin{figure}
\begin{subfigure}{.5\textwidth}
  \centering
  \includegraphics[width=0.75\linewidth]{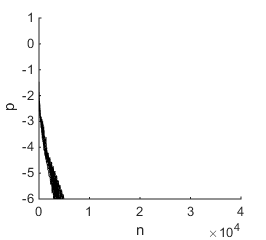}
  \caption{{DAHA-NS}}
  \label{fig:itBest_N=100_1}
\end{subfigure}
\begin{subfigure}{.5\textwidth}
  \centering
  \includegraphics[width=0.75\linewidth]{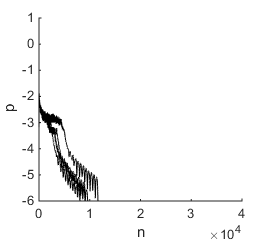}
  \caption{DAHA-S}
  \label{fig:itBest_N=100_2}
\end{subfigure} \\
\begin{subfigure}{.5\textwidth}
  \centering
  \includegraphics[width=0.75\linewidth]{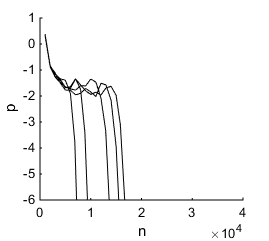}
  \caption{NAP-NS}
    \label{fig:itBest_N=100_3}
\end{subfigure}%
\begin{subfigure}{.5\textwidth}
  \centering
  \includegraphics[width=0.75\linewidth]{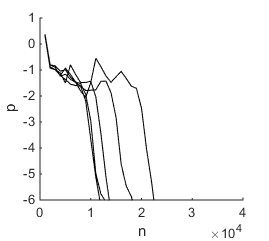}
  \caption{NAP-S}
    \label{fig:itBest_N=100_4}
\end{subfigure} \\
\begin{subfigure}{.5\textwidth}
  \centering
  \includegraphics[width=0.75\linewidth]{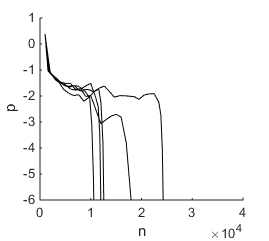}
  \caption{NAV-NS}
    \label{fig:itBest_N=100_5}
\end{subfigure}%
\caption{Relative error on log scale, $p$, as a function of iteration number, $n$, for $5$ randomly generated initial configurations computed by each algorithm for $N=100$ and $\epsilon=10^{-6}$. 
The numerical parameters used are:  
~\ref{fig:itBest_N=100_1} $(\alpha,\beta,c)=(0.07,0.5,2)$,
~\ref{fig:itBest_N=100_2} $(\alpha,\beta,c)=(0.04,0.15,2)$,
~\ref{fig:itBest_N=100_3} $(\alpha,\beta,I_{inner}, \epsilon_{inner})=(0.1,0.16,10^3,10^{-9})$,
~\ref{fig:itBest_N=100_4} $(\alpha,\beta,I_{inner}, \epsilon_{inner})=(0.015,0.026,10^3,10^{-9})$  and
~\ref{fig:itBest_N=100_5} $(\alpha,\beta,I_{inner}, \epsilon_{inner},\tau)=(0.31,41,10^3,10^{-9},0.1)$ .}
\label{fig:itN=100_ini=5} 
\end{figure}

We now consider the five methods, namely, the DAHA-NS, DAHA-S, NAP-NS, NAP-S and NAV-NS. 
In the following we have chosen the numerical parameters $(\alpha, \beta)$ that correspond to a fast convergence of each method for all the $5$ sets of initial configurations.
We plot in Figure~\ref{fig:itN=100_ini=5} the relative error on log scale as a function of iteration number.  Contrarily to the case $N=7$ the DAHA seems to converge faster than the NAP (see Figures~\ref{fig:itBest_N=100_1}-\ref{fig:itBest_N=100_4}). 
The efficiency of the nested algorithms (see Figures~\ref{fig:itBest_N=100_3}-\ref{fig:itBest_N=100_5}) seems to be highly dependent on the initial configuration.

Note that the performance of the methods depends not only on the numerical parameters, but also on the initial configuration. 
In this study, we have only considered initial configurations that are very concentrated around one point.
In the other case, i.e., if  the spheres are initially far away from each other, then all the simulations must be redone and different conclusions may be drawn.

\subsection{Case $N=2000$}

Numerical simulations were successfully performed with the DAHA for $N$ large, up to $N=2000$. 
In order to obtain a faster convergence towards a local minimizer~\cite{PreprintSebastien}, a fourth parameter, $\gamma$, was introduced in the first equation of the original formulation of the DAHA~\eqref{eq:secc_discrete1}-\eqref{eq:secc_discrete2}, yielding,
\begin{numcases}{}
 &$ X_i^{n+1} = 
 \frac{1}{1+c/2} \left( 2X_i^n  - (1-c/2) X_i^{n-1}\right)$\nonumber\\
& \hspace{1cm}$-  \frac{\alpha^2}{1+c/2} [\nabla_{X_i}  W (\XX^n) + \sum\limits_{k,\ell \in \{1,...,N\},\ k<\ell} \lambda_{k\ell}^n \nabla_{X_i} \phi_{k\ell}(\XX^n)]$ \nonumber\\ 
& \hspace{1cm} $- \frac{\gamma^2}{1+c/2} \sum\limits_{k,\ell \in \{1,...,N\},\ k<\ell} \phi_{k\ell}(\XX^n)  \lambda_{k\ell}^n \nabla_{X_i} \phi_{k\ell}(\XX^n), \ i = 1,...,N $\nonumber\\
 &$\lambda_{k\ell}^{n+1} = \max\{0,\lambda_{k\ell}^n + \beta\phi_{k\ell}(\XX^{n+1})\},\ k,\ell = 1,..., N,\ k< \ell. $ \nonumber
\end{numcases}

In Figure~\ref{fig:Ex_N=2000} we present an example of four configurations that were obtained at intermediate steps, namely, $n=1, n=101, n=1001$ and  $n=10001$ for $N=2000$ with the DAHA-S.
The initial configuration was generated from a standard Gaussian distribution.

\begin{figure}
\begin{subfigure}{0.5\textwidth}
 \centering
  \includegraphics[width=1\linewidth]{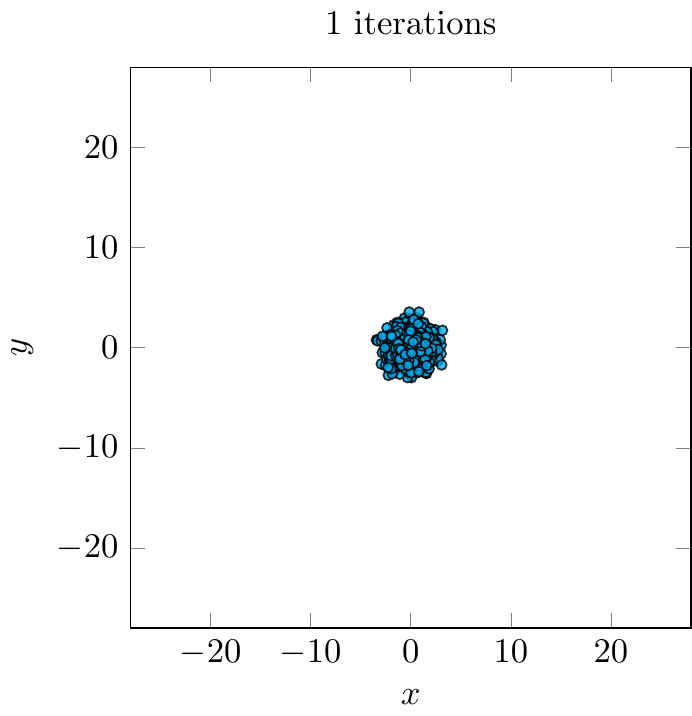}
\end{subfigure} 
\begin{subfigure}{0.5\textwidth}
 \centering
  \includegraphics[width=1\linewidth]{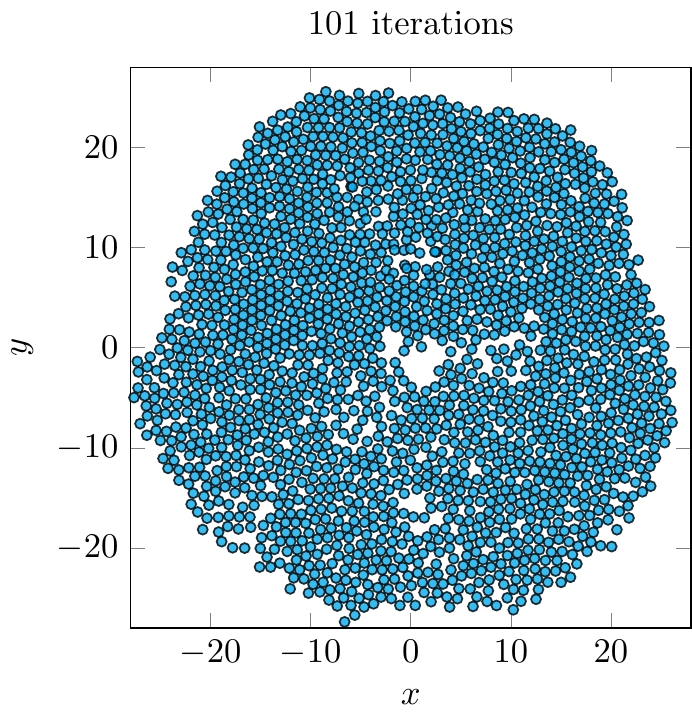}
\end{subfigure} 
\begin{subfigure}{0.5\textwidth}
 \centering
  \includegraphics[width=1\linewidth]{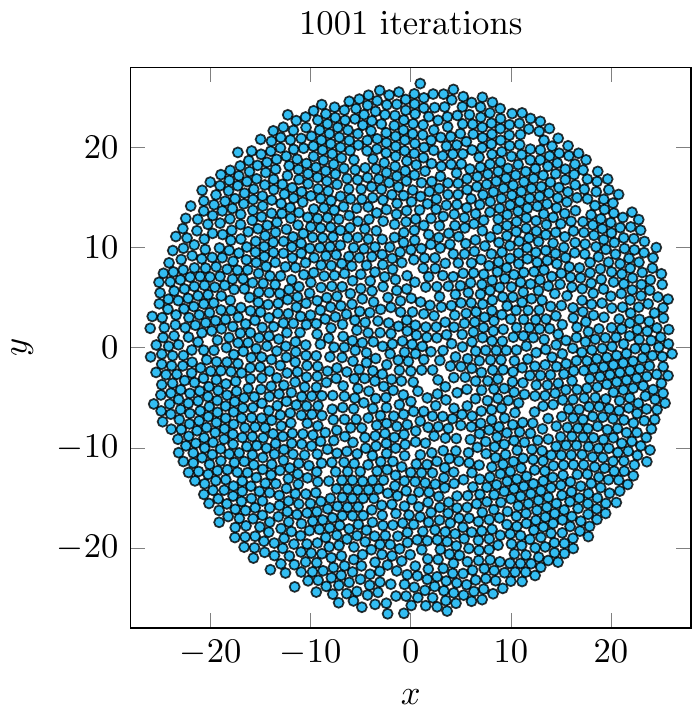}
\end{subfigure}
\begin{subfigure}{0.5\textwidth}
 \centering
  \includegraphics[width=1\linewidth]{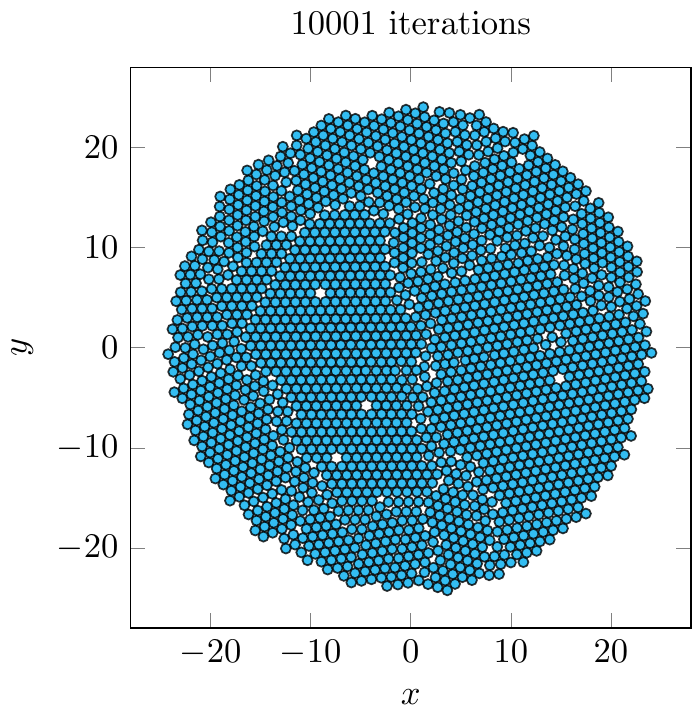}
\end{subfigure}
\caption{Sequence of configurations produced at intermediate steps, namely, $n=1$, $n=101$, $n=1001$ and $n=10001$ with the DAHA-S and for $N=2000$. The numerical parameters used are: $(\alpha, \beta,\gamma, c) = (0.0035\times\sqrt{5}, 1.4,0.7, 2)$.}
\label{fig:Ex_N=2000}
\end{figure}

\subsection{Summary of the results}\label{sec:summary}

\begin{table}[h!]
\centering
\begin{tabular}{l|ll|ll}
& \multicolumn{2}{ c| }{Smooth constraints} & \multicolumn{2}{ c }{Non-smooth constraints} \\ \hline
                              & Analysis & Simulations & Analysis & Simulations\\ \hline
  AHA              & \xmark   &    \xmark   & \cmark   & \xmark   \\
 DAHA       & \cmark   &    \cmark   & \cmark   & \cmark   \\ 
 NAP    & \cmark   &    \cmark   & \cmark   & \cmark   \\ 
  NAV   & \cmark   &   \xmark    & \cmark   & \cmark         \\
%  CAV & \cmark   &  \xmark     & \cmark   & \xmark     \\ 
  \hline
\end{tabular}
\caption{Summary of the results obtained from the analysis for $N=2$ in one spatial dimension and numerical simulations for $N\geq 2$ in two spatial dimensions.}\label{tab:comparison}
\end{table}

We confront in table~\ref{tab:comparison} the results obtained from the theoretical analysis for the case of two spheres ($N = 2$) in one dimension ($b = 1$) and the results of the numerical simulations for the case of $N>2$ spheres in two dimensions ($b=2$).
If the system converges to a non-overlapping configuration within a reasonable number of iterations and for some set of parameters we write \cmark, otherwise we write \xmark.

\section{Conclusions and future work}\label{sec:future}
 
We have deduced a promising algorithm for solving a non-convex minimization problem, which was derived from a multi-step variant of the Arrow-Hurwicz algorithm: the damped Arrow-Hurwicz algorithm.
This algorithm can be seen as a generalization of the AHA, when an additional parameter, $\gamma$, is considered.
In the particular case of packing problems, the DAHA has revealed to perform better for a large number of spheres when compared to other classical algorithms.
However, further studies should be done in order to explore both the advantages and limitations of this method.
In particular, a detailed analysis on the stability of a steady state of the corresponding ODE system for the general case of $N$ spheres in $\RR^b$ is still missing, as well as, the analysis of the numerical stability.
In the present work, the DAHA was assessed in the case of a global potential and  highly dense initial configurations, which we believe to correspond to the worst scenario possible. 
Nevertheless, the results obtained here do not necessarily apply to other types of potentials or initial configurations, and hence similar studies should be conducted for those cases.
In the next work we should consider more general particle systems with different sized spheres or ellipsoids.
The applications of hard-particle systems are vast, for we believe these type of algorithms are going to be very useful in the study of many biological, physical and social systems.

\section*{Acknowledgment}

This work is supported by the National Science Foundation (NSF) under grants DMS-1515592 and RNMS11-07444 (KI-Net), the Engineering and Physical Sciences Research Council (EPSRC) under grant ref. EP/M006883/1 and by the Department of Mathematics, Imperial College, through a Roth PhD studentship. P. D. is on leave from CNRS, Institut de Math\'{e}matiques, Toulouse, France. He acknowledges support from the Royal Society and the Wolfson foundation through a Royal Society Wolfson Research Merit Award.

\newpage
\bibliographystyle{plain}
\bibliography{BiblioPaperNonoverlapping2016}

\end{document}